\documentclass[12pt]{amsart}
\usepackage{amsmath}
\usepackage{amssymb, amsfonts, amsthm}
\usepackage{mathtools}
\usepackage{amssymb}
\usepackage{ifthen}
\usepackage{graphicx}
\usepackage{float}
\usepackage{caption}
\usepackage{subcaption}
\usepackage{cite}
\usepackage{amsfonts}
\usepackage{amscd}
\usepackage{amsxtra}
\usepackage{color}
\usepackage{bm}
\numberwithin{equation}{section}

\newtheorem{theorem}{Theorem}[section]
\newtheorem{lemma}[theorem]{Lemma}
\newtheorem{corollary}[theorem]{Corollary}
\newtheorem{proposition}[theorem]{Proposition}

\theoremstyle{definition}
\newtheorem{definition}[equation]{Definition}

\newtheorem{remark}[equation]{Remark}

\begin{document}
	\title[Multi-dimensional Bohr radii of Banach space valued holomorphic functions]
	{Multi-dimensional Bohr radii of Banach space valued holomorphic functions}

	\subjclass[2020]{ 32A05, 32A10, 46B45.}
	\keywords{Bohr radius, Holomorphic function,  Homogeneous Polynomials, Symmetric $M$-linear map, Banach spaces.}
	
	\author[Shankey Kumar]{Shankey Kumar}
	\address{Shankey Kumar, School of Mathematical Sciences, National Institute of Science Education and Research, Bhubaneswar, An OCC of Homi Bhabha National Institute, Jatni 752050, India}
	\email{shankeygarg93@gmail.com}
	
	\author{Ramesh Manna}
	\address{Ramesh Manna, School of Mathematical Sciences, National Institute of Science Education and Research, Bhubaneswar, An OCC of Homi Bhabha National Institute, Jatni 752050, India}
	\email{rameshmanna@niser.ac.in}

	\begin{abstract}
		In this article, we study the multi-dimensional Bohr radii of holomorphic functions defined on the Banach sequence spaces with values in the Banach spaces. For the case of finite dimensional Banach spaces, we exhibit the exact asymptotic growth of the Bohr radius. To achieve our goal in the finite case, we use $\ell_{p'}$-summability of certain coefficients of a given polynomial in terms of its uniform norm on $\ell_p^n$. The infinite case is handled using the techniques developed in recent years from the work of Defant, Maestre and Schwarting. We crucially use several properties of the symmetric $M$-linear mapping associated with a homogeneous polynomial of degree $M$ in our analysis. Furthermore, we study the bounds of the arithmetic Bohr radius of Banach space-valued holomorphic functions defined on the Banach sequence spaces, which generalises the work of Defant, Maestre, and C. Prengel in this direction.
	\end{abstract}
	
	\maketitle
	
	\section{\bf Introduction}
	In the early nineties, Harald Bohr \cite{B13,B14} made a great contribution to the development of a general theory of the Dirichlet series
	$$
	D(s)=\sum_{n\geq 1} \frac{a_n}{n^s}
	$$
	where $s=\sigma+it$ and $a_n\in \mathbb{C}$. Let us fix $\sigma_0 \in \mathbb{R}$. It is known that the Dirichlet series converges absolutely and uniformly in the half-planes given by $[{\rm Re}(s)>\sigma_0]$ in $\mathbb{C}.$ Bohr's work mainly focused on finding the region of convergence of the above series. To tackle this problem, Bohr discovered a deep connection between the Dirichlet series and the Power series in several variables. For each $n\in \mathbb{N}$, with the help of the prime decomposition $n=p_1^{\alpha_1}\cdots p_r^{\alpha_r}$, let $z=(p_1^{-s},\dots,p_r^{-s})$. Bohr invented the following identities:
	$$
	D(s)=\sum_{n\geq 1} a_n (p_1^{-s})^{\alpha_1} \cdots (p_r^{-s})^{\alpha_r}=\sum a_n z_1^{\alpha_1} \cdots z_r^{\alpha_r},
	$$
	where $(p_k)_{k\in \mathbb{N}}$ denotes the sequence formed by the ordered primes. The above correspondence is known as the Bohr transform. The correspondence between the Dirichlet series and power series in infinite variables brought ideas to the mind of Bohr to ask whether it is possible to compare the absolute value of its coefficients. This problem is nowadays referred to as Bohr's inequality. The application of Bohr's inequality in the Banach algebras provided by Dixon \cite{D95}. After this, many mathematicians were motivated and worked in this direction. To see the work done on the topic of  Bohr's inequality, we refer to a survey article by Abu-Muhanna et al. \cite{Abu-M16}.
	
	\medskip
	Aizenberg, Boas, and Khavinson (see, \cite{A00,B2000,BK97}) generalized the work of Bohr radii in a single variable to several variables. In particular, they introduced and study the Bohr radii for some Reinhardt domains. Examples of Reinhardt's domain are the following:
	$$
	B_{\ell_p^n}:=\{z=(z_1,\dots,z_n)\in\mathbb{C}^n:\|z\|_p<1\}, \,\ 1\leq p \leq \infty,
	$$ 
	where the $p$-norm is given by $\|z\|_p=(\sum_{k=1}^{n}|z_k|^p)^{1/p}$ (with the obvious modification for $p=\infty$).
	
	\medskip
	In \cite{DMS12}, Defant et al. studied the Bohr radius of Banach space-valued holomorphic functions defined on the polydisk $B_{\ell_\infty^n}$. This motivates us to give the following definition of the Bohr radius of Banach space-valued holomorphic functions defined on $B_{\ell_p^n}$, $1\leq p\leq \infty$.

	\begin{definition}\label{def1}
		Let $1\leq p\leq \infty$, $\bm{\lambda}>1$, $n\in\mathbb{N}$, $\mathbb{N}_0=\mathbb{N}\cup \{0\}$, and let $\mathcal{V}: X\to Y$ be a bounded (linear) operator between complex Banach spaces $X$ and $Y$. The {\em $\bm{\lambda}$-Bohr radius of} $\mathcal{V}$, denoted by $K(B_{{\ell}_p^n},\mathcal{V},\bm{\lambda})$ is the supremum of all $r\geq0$ such that the following inequality
		\begin{equation}\label{eq1.2}
			\sup_{z\in rB_{\ell_p^n}} \sum_{\alpha\in \mathbb{N}_0^n} \|\mathcal{V}(c_\alpha) z^\alpha\|_Y \leq \bm{\lambda} \sup_{z\in B_{\ell_p^n}} \big\|f(z)\big\|_X
		\end{equation}
		holds for all holomorphic functions $f(z)= \sum_{\alpha\in \mathbb{N}_0^n} c_\alpha z^\alpha
		$ on $B_{\ell_p^n}$ with values in $Y$.
		Here $c_\alpha \in X$ denote the $\alpha_{th}$ coefficient of monomial series expansion of function $f$, and $z^\alpha={z_1}^{\alpha_1}\cdots {z_n}^{\alpha_n}$ for $\alpha=(\alpha_1, \dots,\alpha_n)\in \mathbb{N}_0^n$ (where $0^0$ is interpreted as $1$). We set $K(B_{{\ell}_p^n},\mathcal{V}):=K(B_{{\ell}_p^n},\mathcal{V},1)$. If $\mathcal{V}$ is an identity operator on $X$ we simply write $K(B_{{\ell}_p^n},X,\bm{\lambda})$ and when $X=\mathbb{C}$,   we use the notation $K(B_{{\ell}_p^n},\bm{\lambda})$. 
	\end{definition}

	Throughout this paper, we will assume that the holomorphic function $f$, defined on $B_{{\ell}_p^n}$, $1\leq p\leq \infty$, is bounded because if $f$ is an unbounded function on $B_{{\ell}_p^n}$, $1\leq p\leq \infty$, then inequality \eqref{eq1.2} is trivial holds.
	
	\medskip 
	A remarkable result by Bohr \cite{B14} states that 
	$$
	K(\mathbb{D})=\frac{1}{3},
	$$
	where $\mathbb{D}:=B_{\ell_p^1}$ denotes the open unit disk.
	More precisely, Bohr’s \cite{B14}, first obtained the radius $1/6$, and later his colleagues Riesz, Schur, and Wiener, independently improved up to $1/3$ and showed that the  radius $1/3$ was sharp.
	
	\medskip

	In \cite{BB04,Bombieri62}, Bombieri and Bourgain has been extended the result of Bohr and studied the general Bohr radius $K(\mathbb{D},\bm{\lambda})$ in terms of $\bm{\lambda} \geq 1$. More precisely, they obtained that
	$$
	K(\mathbb{D},\bm{\lambda})=\frac{1}{3\bm{\lambda}-2\sqrt{2(\bm{\lambda}^2-1)}}, \,\ \mbox{ for all } 1\leq \bm{\lambda} \leq \sqrt{2},
	$$
	and, for $\bm{\lambda}$ close to $\infty$,
	$$
	K(\mathbb{D},\bm{\lambda})\sim\frac{\sqrt{\bm{\lambda}^2-1 }}{\bm{\lambda}}.
	$$
	On the other hand Blasco in \cite[Theorem 1.2]{Blasco} observed that 
	$$
	K(\mathbb{D}, \ell_p^2,1)=0, \mbox{ for every } 1\leq p\leq \infty.
	$$ 
	The above fact force us to consider $\bm{\lambda}>1$ in our definition of the Banach space valued Bohr radii $K(B_{{\ell}_p^n},\mathcal{V},\bm{\lambda})$. In fact, in this article we show that $K(B_{{\ell}_p^n},\mathcal{V},\bm{\lambda})>0$, for every bounded operator $\mathcal{V}:X\to Y$ with $\|\mathcal{V}\|<\bm{\lambda}$ and $\bm{\lambda}>1$.
	
	\medskip
	In the direction of the multi-dimensional Bohr radius Aizenberg, Boas, Dineen, Khavinson, Timoney, Defant and Frerick (see, \cite{A00,B2000,BK97,DF06,DT89}) determined that there is a constant $C\geq 1$ such that the following inequalities hold
	$$
	\frac{1}{C} \Bigg(\frac{\log n}{n\log \log n}\Bigg)^{1-\frac{1}{\min\{p,2\}}}\leq K(B_{{\ell}_p^n}) \leq C \Bigg(\frac{\log n}{n}\Bigg)^{1-\frac{1}{\min\{p,2\}}},
	$$
	for all $1\leq p \leq \infty$ and all $n$.
	Dineen and Timoney in \cite{DT89} renewed the interest in
	multi-dimensional Bohr radii during the study of absolute bases in connection with the spaces of holomorphic functions on
	infinite-dimensional spaces. Khavinson and Boas \cite{BK97} took this work and study the multi-dimensional Bohr radii. They found the above upper bound of
	$K(B_{{\ell}_\infty^n})$ and other interesting results.
	Later, for the remaining cases of $p$, the upper bound of $K(B_{{\ell}_p^n})$ was estimated by Boas in \cite{B2000} by generalising a theorem of Kahane-Salem-Zygmund on random trigonometric polynomials \cite{Kah85}. Further by establishing the link between the multi-dimensional Bohr radii and local Banach space theory (see \cite{DGM03}), Defant and Frerick \cite{DF06} obtained the lower bound of $K(B_{{\ell}_p^n})$.

	\medskip
	The following incredible improvement of the above result is given by Defant and Frerick \cite{DF11} and they provided the existence of a constant $C\geq 1$ such that 
	\begin{equation}\label{eq1.3}
		\frac{1}{C} \Bigg(\frac{\log n}{n}\Bigg)^{1-\frac{1}{\min\{p,2\}}}\leq K(B_{{\ell}_p^n}),
	\end{equation}
	for all $1\leq p \leq \infty$ and all $n$. 
	The important case $p=\infty$ is a consequence of a celebrated fact that ``the Bohnenblust-Hille inequality for homogeneous polynomials is hypercontractive" and was proved in \cite[Theorem 2]{DFOOS11}. Then the case $2\leq p \leq \infty$ was managed using the inequality $K(B_{{\ell}_\infty^n})\leq K(B_{{\ell}_p^n})$ which is true for all $1\leq p \leq \infty$.
	The range $1\leq p<2$ required sophisticated methods from symmetric tensor products and local
	Banach space theory.
	
	\medskip
	The gap between the upper and lower estimates in all of the above papers, as we mentioned, led to many
	efforts to compute the exact asymptotic order of $K(B_{{\ell}_p^n})$ for $1\leq p\leq \infty$. Many mathematical areas such as complex analysis \cite{BDS19,DMP09,LP21}, number theory \cite{BB04,CDGMS15,DFOOS11}, and operator algebra \cite{D95,PPS02} have been enriched during the study of the Bohr radius. The Bohr radius of functions on the Boolean cube was studied in \cite{DMP18,DMP19}. In \cite{GMM2020}, Galicer et al. determined the exact asymptotic growth of the mixed
	$(p,q)$-Bohr radius as $n$ (the number of variables) tends to infinity, for every $1 \leq p \leq q \leq \infty$. Furthermore, many authors have added their contributions to expand the theory of multi-dimensional Bohr's radius (see, e.g., \cite{A05,A07,AAD00,A22,AK22,B2002,BCGMMS21,HHK09,KM22}).

	\medskip
	In \cite[Theorem 4.1]{DMS12}, Defant et al. estimated the asymptotic of $K(B_{{\ell}_\infty^n},X,\bm{\lambda})$, for the finite-dimensional 
	Banach space $X$,
	which is exactly like in \eqref{eq1.3} for $p=\infty$.
	Also, they studied the asymptotic of $K(B_{{\ell}_\infty^n},X,\bm{\lambda})$ for infinite dimensional Banach space $X$ in \cite[Theorem 4.2]{DMS12}. Moreover, they determined the Bohr radii of operators mainly for embeddings in Banach sequence spaces and for bounded operators on some particular type of Banach sequence spaces, like cotype and concave Banach spaces (definitions given in the preliminaries section). All these results are the motivation for this paper. 
	
	\medskip
	
	Our primary interest in this paper is to establish upper and lower estimates for Bohr radii $K(B_{{\ell}_p^n},\mathcal{V},\bm{\lambda})$ of particular operators $\mathcal{V}$ between complex Banach spaces. In general, for every $\|\mathcal{V}\|<\bm{\lambda}$ and $1\leq p \leq \infty$, we determined the following estimates
	$$
	\frac{1}{n^{1-\frac{1}{p}}} \lesssim K(B_{{\ell}_p^n},\mathcal{V},\bm{\lambda}) \lesssim \Bigg(\frac{\log n}{n}\Bigg)^{1-\frac{1}{\min\{p,2\}}}.
	$$
	In particular, for the finite-dimensional Banach space, we determine the following exact asymptotic  estimate of the Bohr radii.
	\begin{theorem} \label{maintheorem}
		Let $\bm{\lambda}>1$. Suppose $X$ is a finite-dimensional Banach space. Then there are constants $C,B>0$ such that the following inequalities
		$$
		C \frac{\bm{\lambda}-1}{2\bm{\lambda}-1} \Bigg(\frac{\log n}{n}\Bigg)^{1-\frac{1}{\min\{p,2\}}}\leq K(B_{{\ell}_p^n},X,\bm{\lambda}) \leq B \bm{\lambda}^2 \Bigg(\frac{\log n}{n}\Bigg)^{1-\frac{1}{\min\{p,2\}}}
		$$
		hold for $n\in \mathbb{N}$ and each $1\leq p \leq \infty$. Here $B$ denotes a universal constant and the constant $C$ depends only on $X$.
	\end{theorem}
	The bounds of the Bohr radii for the general bounded linear operator on $\ell_1$ space to $\ell_q$, $1\leq q< \infty$, we provide in the following theorem.
	\begin{theorem} \label{th1.2}
		Let $1\leq p \leq \infty$ and  $1\leq q<\infty$. Suppose $\mathcal{V}:\ell_1 \to \ell_q$ is any bounded operator. Then, we have
		$$
		\bigg( \frac{\log n}{n} \bigg)^{1-\frac{1}{\max\{2,q\}}}\lesssim K(B_{{\ell}_p^n},\mathcal{V},\bm{\lambda}) \lesssim \Bigg(\frac{\log n}{n}\Bigg)^{1-\frac{1}{\min\{p,2\}}},
		$$
		with constants only depending on $\bm{\lambda}$ and $\mathcal{V}.$
	\end{theorem}
	It is worth here to note that the proof of the above theorem fundamentally uses  the techniques and
	results developed in the last decade. In particular, one of the key arguments for these results is a polynomial reformulation of cotype and concavity.  In fact, the above result follows from the following more general result (for the definitions of cotype and concavity see Section 2.):
	
	\begin{theorem}\label{KMthe3.31}
		Suppose $\mathcal{V}:X \to Y$ is a bounded operator between Banach spaces $X$ and $Y$ with $\|\mathcal{V}\|<\bm{\lambda}$. Then the following holds.
		\begin{enumerate}
			\item Suppose that $X$ or $Y$ is of cotype $q$ with $2\leq q \leq \infty$. Then for every $n$
			$$
			K(B_{{\ell}_p^n},\mathcal{V},\bm{\lambda})\gtrsim \begin{cases}\vspace{0.1cm}
				1,
				& \text{if $p\leq r$,}\\\vspace{0.2cm}
				\cfrac{1}{ n^{\frac{1}{r}-\frac{1}{p}}},
				& \text{if $r<p$}.
			\end{cases}
			$$
			where $r=q/(q-1)$.
			\item Let $2\leq q \leq \infty$. Assume that $Y$ is a $q$-concave Banach lattice and there is a $r \in [1,q)$ such that the operator $\mathcal{V}$ is $(r,1)$-summing. Then, for every $n$
			$$
			K(B_{{\ell}_p^n},\mathcal{V},\bm{\lambda})\gtrsim  \bigg( \frac{\log n}{n} \bigg)^{1-\frac{1}{q}}.
			$$
		\end{enumerate}
	\end{theorem}

	We illustrate this general result regarding the coefficient map of the power series with an example where the coefficient map is given by the Fourier transform. We note that the Fourier transform $\mathcal{F}$ is a linear operator initially defined on $L^1(\mathbb{R}^n)$ and $L^2(\mathbb{R}^n)$. It is well know that the first space $L^1(\mathbb{R}^n)$ is mapped  by $\mathcal{F}$ boundedly with norm $\leq 1$ into $L^{\infty}(\mathbb{R}^n).$ As an application of Plancherel theorem yields that the second space $L^2(\mathbb{R}^n)$ is mapped into itself  by $\mathcal{F}$ with operator norm exactly equal to $1.$ In general, Fourier transform $\mathcal{F}$ has a natural extension to the spaces $L^s(\mathbb{R}^n),~1\leq s\leq 2.$ And  by the Hausdorff-Young inequality, this space $L^s,1\leq s\leq 2$ is mapped by $\mathcal{F}$ boundedly with operator norm $\leq 1$ into dual space $L^{s'}(\mathbb{R}^n),$ where $s,s'$ are H\"{o}lder's conjugate. Since $L^s,~1\leq s\leq 2$ has cotype 2,  we obtain the following:
	\begin{theorem}
		There is constant $C>0$ and for every $n$, we have
		$$
		K(B_{{\ell}_p^n},\mathcal{F},\bm{\lambda})\geq \begin{cases}\vspace{0.1cm}
			C\cfrac{\bm{\lambda}-\|\mathcal{F}\|}{\bm{\lambda}},
			& \text{if $p\leq 2$,}\\\vspace{0.2cm}
			C\cfrac{\bm{\lambda}-\|\mathcal{F}\|}{\bm{\lambda}} \cdot \cfrac{1}{ n^{\frac{1}{r}-\frac{1}{p}}},
			& \text{if $2<p$}.
		\end{cases}
		$$
	\end{theorem}
	We now turn to the arithmetic Bohr radius of Banach spaces, which was first time introduced and studied by Defant et al. in \cite{DMP08} for the scalar case. The concept of the arithmetic Bohr radius plays an important role in studying the upper bound of the multi-dimensional Bohr radius and is also used as a main technical tool in order to derive upper inclusions for domains of convergence by Defant et al. \cite{DMP09}.
	\begin{definition}
		Let $\mathcal{H}(B_{{\ell}_p^n},X)$ be a set of all holomorphic functions on a domain $B_{{\ell}_p^n}$ take value in a complex Banach space $X$. Also, let  $1\leq p\leq \infty$, $n\in\mathbb{N}$, $\bm{\lambda}>1$, and $\mathcal{V}:X\to Y$ be a bounded operator in complex Banach space such that $\|\mathcal{V}\|\leq \bm{\lambda}$.
		Then the {\em$\bm{\lambda}$-arithmetic Bohr radius} of $\mathcal{V}$ is defined as
		\begin{align*}
			A(\mathcal{H}(B_{{\ell}_p^n},X),\mathcal{V},\bm{\lambda}):=  \sup\bigg\{ & \frac{1}{n}\sum_{i=1}^{n}r_i,~ r\in \mathbb{R}_{\geq 0}^n:\sum_{\alpha\in \mathbb{N}_0^n} \|\mathcal{V}(c_\alpha(f))\|_Y ~ r^\alpha\leq \bm{\lambda}\sup_{z\in B_{\ell_p^n}} \big\|f(z)\big\|_X \\ 
			&\mbox{ for all } f\in \mathcal{H}(B_{{\ell}_p^n},X)\bigg\},
		\end{align*}
		where $\mathbb{R}_{\geq 0}^n=\{r=(r_1,\dots,r_n)\in \mathbb{R}^n: r_i\geq 0, 1\leq i \leq n \}$. If $\mathcal{H}(B_{{\ell}_p^n},X)$ is a set of all bounded holomorphic functions defined on $B_{{\ell}_p^n}$ and take value in a complex Banach space $X$ then we set $A(B_{{\ell}_p^n},\mathcal{V},\bm{\lambda}):=A(\mathcal{H}(B_{{\ell}_p^n},X),\mathcal{V},\bm{\lambda})$. If $\mathcal{V}$ is an identity operator on $X$ we set $A(B_{{\ell}_p^n},X,\bm{\lambda}):=A(B_{{\ell}_p^n},\mathcal{V},\bm{\lambda})$. 
	\end{definition}
	Finally, we provide the bounds of the arithmetic Bohr radii for the general bounded linear operator on $\ell_1$ space to $\ell_q$, $1\leq q< \infty$, in the following theorem.
	\begin{theorem}\label{Aoperator}
		Let $1\leq p \leq \infty$ and  $1\leq q<\infty$. Suppose $\mathcal{V}:\ell_1 \to \ell_q$ is any bounded operator. Then, we have
		$$
		\frac{1}{n^{1/p}}\bigg( \frac{\log n}{n} \bigg)^{1-\frac{1}{\max\{2,q\}}}\lesssim A(B_{{\ell}_p^n},\mathcal{V},\bm{\lambda}) \lesssim \cfrac{\big(\log n\big)^{ 1- (1/\min\{p,2\})}}{\displaystyle n^{ \frac{1}{2}+(1/\max\{p,2\})}},
		$$
		with constants only depending on $\bm{\lambda}$ and $\mathcal{V}.$
	\end{theorem}
	
	\begin{remark}
		If we consider $p=q=2$ in the above theorem then interestingly we got the exact asymptotic value of $A(B_{{\ell}_p^n},\mathcal{V},\bm{\lambda})$. 
	\end{remark}
	
	Moreover, for general bounded operator $\mathcal{V}$, with $\|\mathcal{V}\|<\bm{\lambda}$ and $1\leq p \leq \infty$, we have the estimates
	$$
	\frac{1}{n} \lesssim A(B_{{\ell}_p^n},\mathcal{V},\bm{\lambda}) \lesssim \cfrac{\big(\log n\big)^{ 1- (1/\min\{p,2\})}}{\displaystyle n^{ \frac{1}{2}+(1/\max\{p,2\})}},
	$$
	with constants only depending on $\bm{\lambda}$ and $\mathcal{V}$.
	
	The structure of this paper is as follows: In Section 2, we establish some basic properties of the Bohr radius $K(B_{{\ell}_p^n},\mathcal{V},\bm{\lambda})$ and the arithmetic Bohr radius $A(B_{{\ell}_p^n},\mathcal{V},\bm{\lambda})$, for each $1\leq p \leq \infty$, that help to prove our main results. Also, in the same section, we establish a lower bound of the arithmetic Bohr radius in terms of multi-dimensional Bohr radius. In Section 3, we estimate the bounds of $K(B_{{\ell}_p^n},X,\bm{\lambda})$ and $A(B_{{\ell}_p^n},X,\bm{\lambda})$ for particular type of the Banach space $X$. Moreover, we discuss the consequences of these estimates. 
	
	\section{\bf Basic definitions and key estimates}
	Let $\mathcal{P}(\prescript{m}{}{\ell_p^n},X)$ denote the space of all $m$-homogeneous polynomials defined on the $\ell_p^n$ space. Set $|\alpha|:=\alpha_1 +\cdots +\alpha_n$ for $\alpha=(\alpha_1,\dots, \alpha_n)\in \mathbb{N}_0^n$. The polynomial $P\in \mathcal{P}(\prescript{m}{}{\ell_p^n}, X)$ written as the following form
	$$
	P(z_1,\dots,z_n)=\sum_{\alpha\in\Lambda(m,n)} a_\alpha z^\alpha, 
	$$
	where $\Lambda(m,n):=\{\alpha\in\mathbb{N}_0^n:|\alpha|=m\}$ and $a_\alpha \in X$. Alternatively, we can write a polynomial $P\in \mathcal{P}(\prescript{m}{}{\ell_p^n}, X)$ is as
	$$
	P(z_1,\dots,z_n)=\sum_{{\bf j}\in\mathcal{J}(m,n)} c_{\bf j} z_{\bf j}, 
	$$
	where $\mathcal{J}(m,n):=\{{\bf j}=(j_1,\dots,j_m):1\leq j_1\leq\dots \leq j_m\leq n\}$, $z_{\bf j}:=z_{j_1}\dots z_{j_m}$, and $c_{\bf j}\in X$. Then the elements $(z^\alpha)_{\alpha\in\Lambda(m,n)}$ or $(z_{\bf j})_{{\bf j}\in\mathcal{J}(m,n)}$ are referred to as the monomials. Note that the coefficients $c_{\bf j}, a_\alpha$ are related by the relation  $c_{\bf j}=a_\alpha$ with ${\bf j}=(1,\overset{\alpha_1}{\dots},1,\dots,n,\overset{\alpha_n}{\dots},n)$. More precisely there is a one-to-one correspondence between two index sets $\Lambda(m,n)$ and $\mathcal{J}(m,n)$.
	For fixed ${\bf j}\in\mathcal{J}(m,n)$, we define
	$$
	[{\bf j}]=\{{\bf i}=(i_1,\dots,i_m)\in \mathbb{N}^m:(i_{\sigma(1)},\dots,i_{\sigma(m)})={\bf j} \mbox{ some permutation } \sigma \in S_m \}.
	$$ 
	We denote $|{\bf j}|$ the number of elements in $[{\bf j}]$.
	Observe that if the index ${\bf j}$ is $(1,\overset{\alpha_1}{\dots},1,\dots,n,\overset{\alpha_n}{\dots},n)$ then
	$$
	|{\bf j}|=\frac{m!}{\alpha!}.
	$$
	Given a set $\mathcal{J}\subset\mathcal{J}(m,n)$, we define
	$$
	\mathcal{J}^*=\{{\bf j}\in \mathcal{J}(m-1,n): \mbox{ there is } k\geq 1, \ ({\bf j},k)\in\mathcal{J}\}.
	$$
	The above notation was introduced by Bayart et al. in \cite{BDS19}.
	
	Also, we define the notation of the following index set
	$$
	\mathcal{M}(A,n)=\{1,\dots,n\}^A,
	$$
	where $A\subset \mathbb{N}$ and $n\in \mathbb{N}$. So, for $m,n\in\mathbb{N}$, we summarize
	$$
	\mathcal{M}(m,n)=\mathcal{M}(\{1,\dots,m\},n).
	$$
	In $\mathcal{M}(m,n)$, ${\bf i}\sim{\bf j}$ whenever there is a permutation $\sigma \in S_M$ such that $i_{\sigma(k)}=j_k$ for all $1\leq k \leq M$. It is simple that $\mathcal{M}(m,n)=\bigcup_{{\bf j}\in \mathcal{J}(m,n)}[{\bf j}]$. It is simple to observe that $\mbox{card}[{\bf j}_\alpha]=\frac{m!}{\alpha!}$ for every $\alpha \in \Lambda(m,n)$, where $\alpha!=\alpha_1!\cdots \alpha_n!$. 
	
	The absolute value of an element $x$ of a Banach space $X$ is defined by $|x|=x\vee (-x)$.
	A Banach space $X$ is said to be a Banach lattice if it is a vector lattice and satisfies that $\|x\|\leq \|y\|$, for all $x,y \in X$, whenever $|x|\leq |y|$. A Banach lattice $X$ is called $q$-concave, $1\leq q< \infty$, if there is a constant $B>0$ such that the following inequality holds
	$$
	\Bigg(\sum_{k=1}^n \| x_k\|^q\Bigg)^{1/q} \leq B \Bigg\|\Bigg(\sum_{k=1}^n |x_k|^q\Bigg)^{1/q}\Bigg\|,
	$$
	for arbitrarily chosen of finitely many $x_1,\dots,x_n \in X$. 
	
	Here the best such constant $B$ is as usual denoted by $M_q(X)$. Let $2\leq q \leq \infty$. Then, a Banach space $X$ is said to have cotype $q$, if there exists a constant $C>0$ such that the following holds
	$$
	\Bigg(\sum_{k=1}^n \| x_k\|^q\Bigg)^{1/q} \leq C \Bigg(\int_0^1 \Bigg\|\sum_{k=1}^n r_k(t) x_k\Bigg\|^2 dt \Bigg)^{1/2} ,
	$$
	for every choice of vectors $x_1,\dots,x_n \in X$,
	where $r_k$ stands for the $k$th Rademacher function on $[0,1]$. Here the best such constant $C$ is denoted by $C_p(X)$. We set 
	$$
	\mbox{Cot}(X)=\max\{2\leq q \leq \infty| X \mbox{ has cotype } q \}.
	$$
	It is easy to observe that every Banach space has cotype $\infty$, and for $\mbox{Cot}(X)=\infty$ we write $1/\mbox{Cot}(X)=0$. A $q$-concave Banach lattice is cotype $q$, when $q\geq 2$. Conversely, it is well known that each Banach lattice of cotype $2$ is $2$-concave, and for $q>2$, a Banach lattice of cotype $q$ is $r$-concave for all $r>q$. For a detailed study on these topics, we refer \cite{LT77,LT79}.
	
	A bounded linear operator between two Banach spaces $X$ and $Y$ is called $(p,q)$-summing, $1\leq p,q<\infty$, whenever there is a constant $C>0$ such that the following
	\begin{equation}
		\Bigg(\sum_{j=1}^n \|\mathcal{V} (x_j)\|_Y^p\Bigg)^{1/p} \leq C \sup_{x^{*}\in B_{X^{*}}} \Bigg(\sum_{j=1}^n |x^{*}(x_j)|^q\Bigg)^{1/q},
	\end{equation}
	holds for each choice of finitely many $x_1,\cdots,x_n\in X$.
	The best such constant $C$ is as usual denoted by $\pi_{p,q}(\mathcal{V})$. In the case $p=q$, $\mathcal{V}$ is called $p$-summing and we denote $\pi_p(\mathcal{V})$.
	\begin{definition}\label{def2}
		Let  $1\leq p\leq \infty$, $n\in\mathbb{N}$, $\bm{\lambda}>1$, and $\mathcal{V}:X\to Y$ be a bounded operator in complex Banach space such that $\|\mathcal{V}\|\leq \bm{\lambda}$. Then $K_m(B_{{\ell}_p^n},\mathcal{V},\bm{\lambda})$ is defined to be the supremum of all $r\geq 0$ such that the following inequality
		$$
		\sup_{z\in rB_{\ell_p^n}} \sum_{|\alpha|=m} \|\mathcal{V}(c_\alpha) z^\alpha\|_Y \leq \bm{\lambda} \sup_{z\in B_{\ell_p^n}} \big\|P(z)\big\|_X
		$$
		holds for every $M$-homogeneous polynomials $P\in \mathcal{P}(\prescript{m}{}{\ell_p^n}, X)$, where $
		P(z)= \sum_{|\alpha|=m} c_\alpha z^\alpha
		$.
	\end{definition} 
	We can rewrite the above definition as 
	\begin{align}\label{eq2.2}
		K_m(B_{{\ell}_p^n},\mathcal{V},\bm{\lambda})=\sup\Bigg\{&r\geq 0:\sup_{z\in B_{\ell_p^n}} \sum_{|\alpha|=m} \|\mathcal{V}(c_\alpha) z^\alpha\|_Y \leq \frac{\bm{\lambda}}{r^m} \sup_{z\in B_{\ell_p^n}}\big\|P(z)\big\|_X \\ \nonumber &\mbox{ for all }P\in \mathcal{P}(\prescript{m}{}{\ell_p^n}, X) \Bigg\}.
	\end{align}
	Moreover 
	$$
	K_m(B_{\ell_p^n},\mathcal{V},\bm{\lambda})=\bm{\lambda}^{1/m}K_m(B_{\ell_p^n},\mathcal{V},1)
	$$
	and
	$$
	K(B_{\ell_p^n},\mathcal{V},\bm{\lambda})\geq \max\{K(B_{\ell_p^n},X,\bm{\lambda}/\|\mathcal{V}\|),K(B_{\ell_p^n},Y,\bm{\lambda}/\|\mathcal{V}\|)\}.
	$$
	The proof of the following lemma follows exactly like \cite[Lemma 3.2]{DMS12}. So, we are skipping the proof here.
	\begin{lemma}\label{4Klemma1}
		Suppose $\mathcal{V}:X\to Y$ is a bounded operator between complex Banach spaces $X$ and $Y$ with $\|\mathcal{V}\|< \bm{\lambda}$. Then
		\begin{enumerate}
			\item $\frac{\bm{\lambda}-\|\mathcal{V}\|}{2\bm{\lambda}-\|\mathcal{V}\|} \inf_{m\in \mathbb{N}} K_m(B_{{\ell}_p^n},\mathcal{V},\bm{\lambda}) \leq K(B_{{\ell}_p^n},\mathcal{V},\bm{\lambda}) \leq  \inf_{m\in \mathbb{N}} K_m(B_{{\ell}_p^n},\mathcal{V},\bm{\lambda}).$
			\item $\frac{\bm{\lambda}-\|\mathcal{V}\|}{\bm{\lambda}-\|\mathcal{V}\|+1} \inf_{m\in \mathbb{N}} K_m(B_{{\ell}_p^n},\mathcal{V},1) \leq K(B_{{\ell}_p^n},\mathcal{V},\bm{\lambda})
			$.
		\end{enumerate}
	\end{lemma}
	
	The positivity of $\bm{\lambda}$-Bohr radius of $\mathcal{V}$ is provided in the following lemma, which is a generalization of \cite[Proposition 3.3]{DMS12}.
	
	\begin{proposition}\label{Pro2.2}
		Let $1\leq p\leq \infty$. Suppose the operator $\mathcal{V}:X\to Y$ is a non-null bounded operator between complex Banach spaces $X$ and $Y$ with $\|\mathcal{V}\|< \bm{\lambda}$. Then there exists $B>0$ such that for all $n\in \mathbb{N}$ 
		$$
		K(B_{{\ell}_p^n},\mathcal{V},\bm{\lambda})\geq B \frac{1}{n^{1-\frac{1}{p}}},
		$$
		where
		$$
		B=\begin{cases}\vspace{0.1cm}
			\max\bigg\lbrace\cfrac{\bm{\lambda}-\|\mathcal{V}\|}{(\bm{\lambda}-\|\mathcal{V}\|+1)\|\mathcal{V}\|}, \cfrac{\bm{\lambda}-\|\mathcal{V}\|}
			{2\bm{\lambda}-\|\mathcal{V}\|}\bigg\rbrace,
			& \text{if $\|\mathcal{V}\|\geq 1$,}\\\vspace{0.2cm}
			\max\bigg\lbrace\cfrac{\bm{\lambda}-\|\mathcal{V}\|}{\bm{\lambda}-\|\mathcal{V}\|+1}, \cfrac{\bm{\lambda}-\|\mathcal{V}\|}
			{2\bm{\lambda}-\|\mathcal{V}\|}\bigg\rbrace,
			& \text{if $0<\|\mathcal{V}\|< 1$}.
		\end{cases}
		$$
	\end{proposition}
	
	\begin{proof}
		Let $P(z)=\sum_{|\alpha|=m} c_\alpha z^\alpha$ be an $m$-homogeneous polynomial. Then, we have 
		\begin{align*}
			\sup_{z\in \frac{1}{n^{1-\frac{1}{p}}}B_{\ell_p^n}} \sum_{|\alpha|=m} \|\mathcal{V}(c_\alpha) z^\alpha\|_Y &\leq \max_{|\alpha|=m}\|\mathcal{V}(c_\alpha)\|\sup_{z\in \frac{1}{n^{1-\frac{1}{p}}}B_{\ell_p^n}} \sum_{|\alpha|=m} |z^\alpha|\\
			&\leq \|\mathcal{V}\| \max_{|\alpha|=m}\|c_\alpha\|\sup_{z\in \frac{1}{n^{1-\frac{1}{p}}}B_{\ell_p^n}} (|z_1|+\cdots+|z_n|)^m\\
			&\leq \|\mathcal{V}\| \max_{|\alpha|=m}\|c_\alpha\|\sup_{z\in \frac{1}{n^{1-\frac{1}{p}}}B_{\ell_p^n}} (|z_1|^p+\cdots+|z_n|^p)^{m/p} n^{m-\frac{m}{p}}\\
			&\leq \bm{\lambda} \max_{|\alpha|=m}\|c_\alpha\|\leq \bm{\lambda} \|P\|_{\mathcal{P}(\prescript{m}{}{\ell_p^n})},
		\end{align*}
		where $\|P\|_{\mathcal{P}(\prescript{m}{}{\ell_p^n})}=\sup_{z\in B_{\ell_p^n}}|P(z)|$.
		This implies that for all $m$
		$$
		K_m(B_{{\ell}_p^n},\mathcal{V},\bm{\lambda})\geq \frac{1}{n^{1-\frac{1}{p}}}.
		$$
		Hence, by Lemma \ref{4Klemma1}$(1)$ we have
		$$
		K(B_{{\ell}_p^n},\mathcal{V},\bm{\lambda})\geq \frac{1}{n^{1-\frac{1}{p}}}\cfrac{\bm{\lambda}-\|\mathcal{V}\|}
		{2\bm{\lambda}-\|\mathcal{V}\|}.
		$$
		By similar argument as above, for every $m$-homogeneous polynomial $P(z)=\sum_{|\alpha|=m} c_\alpha z^\alpha$ we get that
		$$
		\sup_{z\in \frac{1}{\sqrt[m]{\|\mathcal{V}\|}n^{1-\frac{1}{p}}}B_{\ell_p^n}} \sum_{|\alpha|=m} \|\mathcal{V}(c_\alpha) z^\alpha\|_Y=\frac{1}{\|\mathcal{V}\|}\sup_{z\in \frac{1}{n^{1-\frac{1}{p}}}B_{\ell_p^n}} \sum_{|\alpha|=m} \|\mathcal{V}(c_\alpha) z^\alpha\|_Y\leq \|P\|_{\mathcal{P}(\prescript{m}{}{\ell_p^n})}.
		$$
		Hence for every $m$ we have
		$$
		K_m(B_{{\ell}_p^n},\mathcal{V},1)\geq \cfrac{\bm{\lambda}-\|\mathcal{V}\|}{(\bm{\lambda}-\|\mathcal{V}\|+1)\sqrt[m]{\|\mathcal{V}\|}}\frac{1}{n^{1-\frac{1}{p}}}.
		$$
		Now, Lemma \ref{4Klemma1}$(2)$ yields
		$$
		K(B_{{\ell}_p^n},\mathcal{V},1)\geq \cfrac{\bm{\lambda}-\|\mathcal{V}\|}{(\bm{\lambda}-\|\mathcal{V}\|+1)\|\mathcal{V}\|}\frac{1}{n^{1-\frac{1}{p}}}, \,\ \mbox{ if } \|\mathcal{V}\|\geq 1
		$$
		and
		$$
		K(B_{{\ell}_p^n},\mathcal{V},1)\geq \cfrac{\bm{\lambda}-\|\mathcal{V}\|}{\bm{\lambda}-\|\mathcal{V}\|+1}\frac{1}{n^{1-\frac{1}{p}}}, \,\ \mbox{ if } 0<\|\mathcal{V}\|< 1.
		$$
		This concludes the result.
	\end{proof}
	Moreover, for the identity operator, Proposition \ref{Pro2.2} provides the following results for the Bohr radius of Banach spaces.
	\begin{corollary}\label{corollary2.3}
		Suppose $X$ is a complex Banach space and $\bm{\lambda}>1$, then for each $n\in\mathbb{N}$ we have
		$$
		K(B_{{\ell}_p^n},X,\bm{\lambda})\geq \cfrac{\bm{\lambda}-1}{\bm{\lambda}} \cdot \frac{1}{n^{1-\frac{1}{p}}}.
		$$
	\end{corollary}
	Here the interesting fact is that we get the trivial lower bound $K(B_{{\ell}_p^n},X,\bm{\lambda})\geq 0$ for $\bm{\lambda}=1$.
	In the following lemma, we provide a lower estimate of $ K_m(B_{{\ell}_p^n},X, \bm{\lambda})$, $1\leq p \leq \infty$, in terms of $ K_m(B_{{\ell}_\infty^n},X, \bm{\lambda})$.  
	\begin{lemma}\label{4Klemma2}
		Let $1\leq p \leq \infty$ and $\bm{\lambda} > 1$. Suppose $X$ is a Banach space. Then for each $m \in \mathbb{N}$
		$$
		B \frac{\bm{\lambda}-1}{2\bm{\lambda}-1} \sqrt{\frac{\log n}{n}}\leq K_m(B_{{\ell}_\infty^n},X, \bm{\lambda})\leq K_m(B_{{\ell}_p^n},X, \bm{\lambda}), 
		$$
		where $B$ depends only on $X$.  
	\end{lemma}
	\begin{proof} As we know from the proof of \cite[Theorem 4.1]{DMS12} that
		$$
		B \frac{\bm{\lambda}-1}{2\bm{\lambda}-1} \sqrt{\frac{\log n}{n}}\leq K_m(B_{{\ell}_\infty^n},X, \bm{\lambda}), \mbox{ for every } m \in \mathbb{N}.
		$$
		So, to complete the proof, we need to only prove that 
		$$
		K_m(B_{{\ell}_\infty^n},X, \bm{\lambda})\leq K_m(B_{{\ell}_p^n},X, \bm{\lambda}). 
		$$
		For every $z\in B_{\ell_p^n}$, we have
		\begin{align*}
			\sum_{|\alpha|=m}\|c_\alpha z^\alpha\| 
			& = \frac{1}{(K_m(B_{{\ell}_\infty^n},X, \bm{\lambda}))^m} \sup_{w\in  B_{\ell_\infty^n}}\sum_{|\alpha|=m}\|c_\alpha z^\alpha w^\alpha (K_m(B_{{\ell}_\infty^n},X, \bm{\lambda}))^m\|  \\
			& \leq \frac{\bm{\lambda}}{(K_m(B_{{\ell}_\infty^n},X, \bm{\lambda}))^m}  \sup_{w\in  B_{\ell_\infty^n}}\bigg\|\sum_{|\alpha|=m} c_\alpha z^\alpha w^\alpha \bigg \| \\
			&\leq \frac{\bm{\lambda}}{(K_m(B_{{\ell}_\infty^n},X, \bm{\lambda}))^m} \sup_{z\in B_{\ell_p^n}}\bigg\|\sum_{|\alpha|=m} c_\alpha z^\alpha \bigg \|.
		\end{align*}
		This concludes the result.
	\end{proof}
	The next result provides the estimate of $A(B_{{\ell}_p^n},\mathcal{V},\bm{\lambda})$ in terms of $A\big(\bigcup_{m=1}^{\infty}\mathcal{P}(\prescript{m}{}{\ell_p^n},X),\mathcal{V},\bm{\lambda})$. 
	\begin{proposition}\label{Prop2.5}
		Suppose $\mathcal{V}:X\to Y$ is a bounded operator between complex Banach spaces $X$ and $Y$ with $\|\mathcal{V}\|< \bm{\lambda}$. Then
		\begin{enumerate}
			\item $\frac{\bm{\lambda}-\|\mathcal{V}\|}{2\bm{\lambda}-\|\mathcal{V}\|}  A\big(\bigcup_{m=1}^{\infty}\mathcal{P}(\prescript{m}{}{\ell_p^n},X),\mathcal{V},\bm{\lambda}) \leq A(B_{{\ell}_p^n},\mathcal{V},\bm{\lambda}) \leq  A\big(\bigcup_{m=1}^{\infty}\mathcal{P}(\prescript{m}{}{\ell_p^n},X),\mathcal{V},\bm{\lambda}).$
			\item $\frac{\bm{\lambda}-\|\mathcal{V}\|}{\bm{\lambda}-\|\mathcal{V}\|+1} A\big(\bigcup_{m=1}^{\infty}\mathcal{P}(\prescript{m}{}{\ell_p^n},X),\mathcal{V},1) \leq A(B_{{\ell}_p^n},\mathcal{V},\bm{\lambda})
			$.
		\end{enumerate}
	\end{proposition}
	\begin{proof} The right inequality of $(1)$ is easy to understand. Therefore, we focus on the left inequality. Consider $r\in \mathbb{R}_{\geq 0}^n$ such that 
		$$
		\sum_{|\alpha|=m} \big\| \mathcal{V}(c_\alpha(P))  \big\|r^\alpha \leq \bm{\lambda} \sup_{z\in B_{\ell_p^n}}\|P(z)\|
		$$
		for all $m$ and $P\in \mathcal{P}(\prescript{m}{}{\ell_p^n},X)$. Let  $f(z)= \sum_{\alpha\in \mathbb{N}_0^n} c_\alpha z^\alpha$ be a holomorphic function on $B_{\ell_p^n}$ and takes value in $X$.
		Then we obtain 
		\begin{align*}
			\sum_{\alpha\in \mathbb{N}_0^n} \big\| \mathcal{V}(c_\alpha)  \big\|\bigg(\frac{\bm{\lambda}-\|\mathcal{V}\|}{2\bm{\lambda}-\|\mathcal{V}\|}r \bigg)^\alpha &= \|\mathcal{V}(c_0)\|+ \sum_{m=1}^{\infty}\bigg(\frac{\bm{\lambda}-\|\mathcal{V}\|}{2\bm{\lambda}-\|\mathcal{V}\|}\bigg)^m\sum_{|\alpha|=m} \| \mathcal{V}(c_\alpha)\| r^\alpha\\
			&\leq \|\mathcal{V}\|\|c_0\|+ \bm{\lambda}\sum_{m=1}^{\infty}\bigg(\frac{\bm{\lambda}-\|\mathcal{V}\|}{2\bm{\lambda}-\|\mathcal{V}\|}\bigg)^m \sup_{z\in B_{\ell_p^n}}\bigg\|\sum_{|\alpha|=m} c_\alpha z^\alpha \bigg\|.
		\end{align*}
		Thus, by the Cauchy-Riemann inequalities, we have 
		$$
		\sum_{\alpha\in \mathbb{N}_0^n} \big\| \mathcal{V}(c_\alpha)  \big\|\bigg(\frac{\bm{\lambda}-\|\mathcal{V}\|}{2\bm{\lambda}-\|\mathcal{V}\|}r \bigg)^\alpha \leq \bm{\lambda} \sup_{z\in B_{\ell_p^n}}\|f(z)\|.
		$$
		For the $(2)$ inequality, similar to the above process by taking $r\in \mathbb{R}_{\geq 0}^n$ such that 
		$$
		\sum_{|\alpha|=m} \big\| \mathcal{V}(c_\alpha(P))  \big\|r^\alpha \leq \sup_{z\in B_{\ell_p^n}}\|P(z)\|
		$$
		for all $m$ and $P\in \mathcal{P}(\prescript{m}{}{\ell_p^n},X)$.
		Then we get
		$$
		\sum_{\alpha\in \mathbb{N}_0^n} \big\| \mathcal{V}(c_\alpha)  \big\|\bigg(\frac{\bm{\lambda}-\|\mathcal{V}\|}{\bm{\lambda}-\|\mathcal{V}\|+1}r\bigg)^\alpha\leq \bm{\lambda} \sup_{z\in B_{\ell_p^n}}\|f(z)\|. 
		$$
		This completes the proof.
	\end{proof}
	In the following theorem, we will provide an important relation between the multi-dimensional and arithmetic Bohr radius of $\mathcal{V}$ (scalar case already discussed in \cite[Lemma 4.3]{DMP08}).
	\begin{lemma}\label{KMthe1.4}
		Let $\mathcal{V}:X\to Y$ be a bounded operator in complex Banach space such that $\|\mathcal{V}\|\leq \bm{\lambda}$. Then	for every $1\leq\bm{\lambda}<\infty$, $1\leq p \leq \infty$ and $n$, we have
		$$
		A(B_{\ell_p^n},\mathcal{V},\bm{\lambda})\geq \frac{K(B_{\ell_p^n},\mathcal{V},\bm{\lambda})}{n^{1/p}}.
		$$
	\end{lemma} 
	\begin{proof}
		An easy calculation gives 
		$$
		\sup_{z\in B_{\ell_p^n}}\|z\|_1= n^{1-\frac{1}{p}}.
		$$
		Thus for $0<\epsilon<K(B_{\ell_p^n},\mathcal{V},\bm{\lambda})$ we can find an element $z_0:=(z_1^0,\dots,z_n^0) \in B_{\ell_p^n}$ such that 
		$$
		\|z_0\|_1 \geq n^{1-\frac{1}{p}} - \epsilon.
		$$
		Let $t:=K(B_{\ell_p^n},\mathcal{V},\bm{\lambda})-\epsilon$ and $r:=t(|z_1^0|,\dots,|z_n^0|)$. Then 
		$$
		\sum_{\alpha\in \mathbb{N}_0^n} \|\mathcal{V}(c_\alpha(f))\|r^{\alpha}=\sum_{\alpha\in \mathbb{N}_0^n} \|\mathcal{V}(c_\alpha(f))t^{|\alpha|}z_0^{\alpha}\|\leq \big\|\sum_{\alpha\in \mathbb{N}_0^n} \|\mathcal{V}(c_\alpha(f)) z^{\alpha}\| \big\|_{tB_{\ell_p^n}}.
		$$
		Since $t<K(B_{\ell_p^n},\mathcal{V},\bm{\lambda})$ we have 
		$$
		\sum_{\alpha\in \mathbb{N}_0^n} \|\mathcal{V}(c_\alpha(f))\|r^{\alpha} \leq \bm{\lambda} \|f\|_{B_{\ell_p^n}}.
		$$
		It leads to the fact that 
		$$
		\frac{1}{n}\sum_{i=1}^{n} r_i= \frac{K(B_{\ell_p^n},\mathcal{V},\bm{\lambda})-\epsilon}{n}\|z_0\|_1 \geq \frac{K(B_{\ell_p^n},\mathcal{V},\bm{\lambda})-\epsilon}{n}(n^{1-\frac{1}{p}} - \epsilon).
		$$
		The above estimate gives the desired inequality.
	\end{proof}
	As a consequence of Proposition \ref{Pro2.2} and Lemma \ref{KMthe1.4} we obtain the following result.
	\begin{proposition}
		Let $1\leq p\leq \infty$. Suppose the operator $\mathcal{V}:X\to Y$ is a non-null bounded operator between complex Banach spaces $X$ and $Y$ with $\|\mathcal{V}\|< \bm{\lambda}$. Then there exists $B>0$ such that for all $n\in \mathbb{N}$ 
		$$
		A(B_{{\ell}_p^n},\mathcal{V},\bm{\lambda})\geq B \frac{1}{n},
		$$
		where $B$ is defined as in Proposition \ref{Pro2.2}.
	\end{proposition}
	Now, we are ready to prove our main results.
	\section{\bf Proof of the main results on multi-dimensional Bohr radii}
	We begin this section with the following result, due to Bayart et al. \cite[Lemma 3.5]{BDS19} (see also \cite[Lemma 19.11]{DGMS19} and \cite{GMM2020}), which plays a key role to prove Theorem \ref{maintheorem}. 
	\begin{lemma}\label{4Klemma3}
		Let $1\leq p \leq \infty$ and $P\in \mathcal{P}(\prescript{m}{}{\ell_p^n}, \mathbb{C})$. Then, for any ${\bf j}\in\mathcal{J}(m-1,n)$,
		$$
		\Bigg(\sum_{k=j_{m-1}}^n |c_{({\bf j},k)} (P)|^{p'}\Bigg)^{1/p'} \leq m e^{1+\frac{m-1}{p}} |{\bf j}|^{1/p} \|P\|_{\mathcal{P}(\prescript{m}{}{\ell_p^n})},
		$$
		where $p'=p/(p-1)$ and $\|P\|_{\mathcal{P}(\prescript{m}{}{\ell_p^n})}=\sup_{z\in B_{\ell_p^n}}|P(z)|$.	
	\end{lemma}
	\begin{proof}[Proof of Theorem \ref{maintheorem}]
		We start the proof with a following simple observation 
		$$
		K(B_{{\ell}_p^n},X,\bm{\lambda})\leq K(B_{{\ell}_p^n},\mathbb{C},\bm{\lambda}).
		$$
		The proof of \cite[Lemma 4.3]{DMP08} provides that
		$$
		K(B_{{\ell}_p^n},\mathbb{C},\bm{\lambda})\leq n^{1/p}A(B_{{\ell}_p^n},\bm{\lambda}), \, \mbox{ for all } \bm{\lambda}\geq 1 \mbox{ and } n\in \mathbb{N},
		$$
		here $A(B_{{\ell}_p^n},\bm{\lambda})$ is called as the arithmetic Bohr radius. On the other hand, form \cite[Theorem 4.1]{DMP08} we know that there exits a constant $B>0$ such that for all $n\in\mathbb{N}$
		$$
		A(B_{{\ell}_p^n},\bm{\lambda})\leq B \bm{\lambda}^{\frac{2}{\log n}} \cfrac{\big(\log n\big)^{ 1- (1/\min\{p,2\})}}{\displaystyle n^{ \frac{1}{2}+(1/\max\{p,2\})}}.
		$$
		Thus we have, for all $\bm{\lambda}\geq 1$ and $n\in\mathbb{N}$,
		$$
		K(B_{{\ell}_p^n},X,\bm{\lambda})\leq  B \bm{\lambda}^{\frac{2}{\log n}} \Bigg(\cfrac{\log n}{n}  \Bigg)^{1-\frac{1}{\min\{p,2\}}}.
		$$
		Now to obtain the lower bound, by Lemma \ref{4Klemma1}, it suffices to prove that 
		$$
		C(X)\Bigg(\cfrac{\log n}{n}  \Bigg)^{1-\frac{1}{\min\{p,2\}}}\leq\inf_{m\in \mathbb{N}} K_m(B_{{\ell}_p^n},X,\bm{\lambda}),
		$$
		for some constant $C(X)>0$. To proceed in this proof, we fix $P\in \mathcal{P}(\prescript{m}{}{\ell_p^n})$  and $u \in \ell_p^n$. Then by taking help of Lemma \ref{4Klemma3}, we obtain for any ${\bf j} \in \mathcal{J}(m,n)^{*}$
		\begin{align}\label{eq3.1}
			\Bigg(\sum_{k:\ ({\bf j},k)\in \mathcal{J}(m,n)} |c_{({\bf j},k)} (P)|^{p'}\Bigg)^{1/p'}&=\Bigg(\sum_{k=j_{m-1}}^n |c_{({\bf j},k)} (P)|^{p'}\Bigg)^{1/p'}\\ &\leq m e^{1+\frac{m-1}{p}} |{\bf j}|^{1/p} \|P\|_{\mathcal{P}(\prescript{m}{}{\ell_p^n})}.\nonumber
		\end{align}
		Since $X$ is  finite dimensional, there exist a constant $C>0$ such that 
		\begin{equation}\label{eq3.2}
			\Bigg(\sum_{j=1}^h \|x_j\|^s\Bigg)^{1/s} \leq C \sup_{x^{*}\in B_{X^{*}}} \Bigg(\sum_{j=1}^h |x^{*}(x_j)|^s\Bigg)^{1/s},
		\end{equation}
		holds for every $s\geq 1$ and for every $x_1 \dots x_h \in X$.
		In fact the identity operator $I$ on finite dimensional Banach spaces $X$ is $1$-summing and thus the operator $I$ is $s$-summing for every $s\geq 1$ with $\pi_s(I)\leq \pi_1(I)$, see for instance \cite[Theorem 10.4]{DJT95}.
		
		It is easy to see that
		$$
		\sum_{{\bf j}\in \mathcal{J}(m,n)} \|c_{\bf j} (P)\||u_{\bf j}|=\sum_{{\bf j}\in \mathcal{J}(m,n)^{*}} \Bigg( \sum_{k:\ ({\bf j},k)\in \mathcal{J}(m,n)} \|c_{({\bf j},k)} (P)\||u_{\bf j}||u_k|\Bigg)
		$$
		Now applying H\"{o}lder's inequality (two times), \eqref{eq3.1} and \eqref{eq3.2}, we obtain 
		\begin{align*}
			\sum_{{\bf j}\in \mathcal{J}(m,n)} \|c_{\bf j} (P)\| \ |u_{\bf j}|&\leq \sum_{{\bf j}\in \mathcal{J}(m,n)^{*}} |u_{\bf j}| \ \Bigg( \sum_{k:\ ({\bf j},k)\in \mathcal{J}(m,n)} \|c_{({\bf j},k)} (P)\|^{p'}\Bigg)^{1/p'} \Big(\sum_{k} |u_k|^p \Big)^{1/p}\\
			&\leq \pi_1(I) \sum_{{\bf j}\in \mathcal{J}(m,n)^{*}} |u_{\bf j}| \ \sup_{x^{*}\in B_{X^{*}}}\Bigg( \sum_{k:\ ({\bf j},k)\in \mathcal{J}(m,n)}^n |x^{*}(c_{({\bf j},k))} (P))|^{p'}\Bigg)^{1/p'} \|u\|_p\\
			&\leq \pi_1(I) m e^{1+\frac{m-1}{p}} \|u\|_p  \sup_{x^{*}\in B_{X^{*}}}\|x^{*}(P)\|_{\mathcal{P}(\prescript{m}{}{\ell_p^n})} \sum_{{\bf j}\in \mathcal{J}(m,n)^{*}} |{\bf j}|^{1/p} |u_{\bf j}| \\
			& \leq \pi_1(I) m e^{1+\frac{m-1}{p}}  \|u\|_p \ \sup_{z\in B_{\ell_p^n}}\|P(z)\| \Bigg(\sum_{{\bf j}\in \mathcal{J}(m,n)^{*}} |{\bf j}| \ |u_{\bf j}|^p\Bigg)^{1/p}\Bigg(\sum_{{\bf j}\in \mathcal{J}(m,n)^{*}} 1\Bigg)^{1/p'}\\
			& \leq \pi_1(I) m e^{1+\frac{m-1}{p}}  \|u\|_p \ \sup_{z\in B_{\ell_p^n}}\|P(z)\| \Bigg(\sum_{{\bf j}\in \mathcal{J}(m-1,n)} |{\bf j}| \ |u_{\bf j}|^p\Bigg)^{1/p}\Bigg(\sum_{{\bf j}\in \mathcal{J}(m-1,n)} 1\Bigg)^{1/p'}.
		\end{align*}
		Since 
		$$
		\Bigg(\sum_{{\bf j}\in \mathcal{J}(m-1,n)} |{\bf j}| \ |u_{\bf j}|^p\Bigg)^{1/p} =\|u\|_p^{m-1},
		$$
		hence we obtain
		$$
		\sum_{{\bf j}\in \mathcal{J}(m,n)} \|c_{\bf j} (P)\| \ |u_{\bf j}| \leq  \pi_1(I) m e^{1+\frac{m-1}{p}}  \|u\|_p^m \ \sup_{z\in B_{\ell_p^n}}\|P(z)\| \Bigg(\sum_{{\bf j}\in \mathcal{J}(m-1,n)} 1\Bigg)^{1/p'}.
		$$
		If we use the fact that the number of elements in the index set $\mathcal{J}(m-1,n)$ is 
		$$
		\frac{(n+m-2)!}{(n-1)!(m-1)!},
		$$
		then we get
		$$
		\sum_{{\bf j}\in \mathcal{J}(m,n)} \|c_{\bf j} (P)\| \ |u_{\bf j}| \leq \pi_1(I) \Bigg(\frac{(n+m-2)!}{(n-1)!(m-1)!}\Bigg)^{1/p'} m e^{1+\frac{m-1}{p}}  \|u\|_p^m \ \sup_{z\in B_{\ell_p^n}}\|P(z)\|.
		$$
		It is not difficult to observe that 
		$$
		\frac{(n+m-2)!}{(n-1)!(m-1)!} \leq \frac{(n+m)^{m-1}}{(m-1)!} \leq \frac{e^m}{m^{m-1}} (n+m)^{m-1}=e^m \bigg(1+\frac{n}{m} \bigg)^{m-1}.
		$$
		It leads that
		$$
		\sum_{{\bf j}\in \mathcal{J}(m,n)} \|c_{\bf j} (P)\| \ |u_{\bf j}|\leq \pi_1(I) \ m \ \bigg(1+\frac{n}{m} \bigg)^{\frac{m-1}{p'}}  e^{\big(m+1-\frac{1}{p}\big)} \ \|u\|_p^m \ \sup_{z\in B_{\ell_p^n}}\|P(z)\|.
		$$
		A simple observation gives that
		\begin{equation*}
			\bigg(1+\frac{n}{m} \bigg)^{\frac{m-1}{m}} \leq 2^{\frac{m-1}{m}} \max \bigg\{1, \bigg(\frac{n}{m} \bigg)^{\frac{m-1}{m}} \bigg\} \leq 2 \max \bigg\{1, \frac{m^{1/m}n}{m n^{1/m}} \bigg\}.
		\end{equation*}
		Note that $x\to xn^{1/x}$ is decreasing upto $x=\log n$ and increasing thereafter. Thus we get
		$$
		\bigg(1+\frac{n}{m} \bigg)^{\frac{m-1}{m}} \leq 2\frac{n}{\log n}.
		$$
		Hence, by \eqref{eq2.2},
		$$
		K_m(B_{{\ell}_p^n},X,\bm{\lambda})\geq \bigg(\frac{\bm{\lambda}}{m\ \pi_1(I)}\bigg)^{1/m} \frac{1}{2^{\frac{1}{p'}}} \frac{1}{e^{\big(1+\frac{1}{mp'}\big)}} \Bigg(\cfrac{\log n}{n}  \Bigg)^{1-\frac{1}{p}}.
		$$
		Then, there exists a constant $C(X)$, only depending on $X$, such that
		$$
		\inf_{m\in \mathbb{N}} K_m(B_{{\ell}_p^n},X,\bm{\lambda})\geq C(X)\Bigg(\cfrac{\log n}{n}  \Bigg)^{\frac{1}{p'}}.
		$$
		Finally, using Lemma \ref{4Klemma2}, we get 
		$$
		\inf_{m\in \mathbb{N}} K_m(B_{{\ell}_p^n},X,\bm{\lambda})\geq C(X) \Bigg(\cfrac{\log n}{n}  \Bigg)^{1-\frac{1}{\min\{p,2\}}}.
		$$
		This concludes the proof.
	\end{proof}
	It is easy to see that the above asymptotic of $K(B_{{\ell}_p^n},X,\bm{\lambda})$, for each $1\leq p \leq \infty$, is exactly like in \eqref{eq1.3}. To obtain the exact asymptotic estimate of $A(B_{\ell_p^n},X,\bm{\lambda})$, the proof, together with Theorem \ref{maintheorem} and Lemma \ref{KMthe1.4}, follows exactly as in \cite[Theorem 4.1]{DMP08} (see also, \cite[Remark 2]{DF11} and \cite[Theorem 19.13]{DGMS19}).
	\begin{theorem}
		Let $1\leq p \leq \infty$. Suppose that $X$ is a finite-dimensional complex Banach space. Then there are constants $C,B>0$ such that for every $n\in \mathbb{N}$
		$$
		C \frac{\bm{\lambda}-1}{2\bm{\lambda}-1}\frac{  \big( \log n\big)^{1-(1/\min\{p,2\})}}{\displaystyle n^{\frac{1}{2}+(1/\max\{p,2\})}} \leq A(B_{\ell_p^n},X,\bm{\lambda}) \leq B \bm{\lambda}^2 \cfrac{\big(\log n\big)^{ 1- (1/\min\{p,2\})}}{\displaystyle n^{ \frac{1}{2}+(1/\max\{p,2\})}},
		$$
		here $B$ is a universal constant and $C$ depends only on $X$.
	\end{theorem}

	Now we focus on infinite dimensional cases. In this case, we shall see that in the expression of the lower bound of Bohr radius, $\log$-term is not present. In Corollary \ref{corollary2.3}, the lower bound of  $K(B_{{\ell}_p^n},X,\bm{\lambda})$, when $X$ is a Banach space without finite cotype, was already obtained. Now in the following result, we assume that $X$ has cotype $q$. Given an $m$-homogeneous polynomial $P:\mathbb{C}^n \to X$, $P(z)=\sum_{|\alpha|=m} c_\alpha z^\alpha$, we denote by $A:\mathbb{C}^n \times\cdots\times \mathbb{C}^n \to X$ the unique symmetric $m$-linear mapping associated to $P$. In \cite[Theorem 3.2]{Bom04}, Bombal has been given the following inequality
	\begin{equation}\label{KMeq3.3}
		\bigg(\sum_{{\bf i}\in \mathcal{M}(m,n)} \|A(e_{i_1},\cdots,e_{i_m})\|^q \bigg)^{\frac{1}{q}} \leq C_q(X)^m \|A\|_{B_{{\ell}_p^n}},
	\end{equation}
	which will use as a tool to prove the following theorem. For the case $p=\infty$, the following result is already obtained in \cite[Theorem 4.2]{DMS12}.
	\begin{theorem} \label{Bqcotype}
		Let $1\leq p \leq \infty$. Suppose $X$ is an infinite dimensional complex Banach space of cotype $q$ with $2\leq q\leq \infty$. Then we have
		$$
		K(B_{{\ell}_p^n},X,\bm{\lambda})\leq\begin{cases}\vspace{0.1cm}
			\cfrac{\bm{\lambda}}{n^{1-\frac{1}{p}}},
			& \text{if $p\leq\mbox{Cot}(X)$,}\\\vspace{0.2cm}
			\cfrac{\bm{\lambda}}{n^{1-\frac{1}{\text{Cot}(X)}}},
			& \text{if $\mbox{Cot}(X)<p$},
		\end{cases}
		$$
		and
		$$
		K(B_{{\ell}_p^n},X,\bm{\lambda})\geq \begin{cases}\vspace{0.1cm}
			\cfrac{\bm{\lambda}-1}{\bm{\lambda}} \cdot \cfrac{1}{e},
			& \text{if $p\leq r$,}\\\vspace{0.2cm}
			\cfrac{\bm{\lambda}-1}{\bm{\lambda} e C_q(X)} \cdot \cfrac{1}{n^{\frac{1}{r}-\frac{1}{p}} },
			& \text{if $r<p$},
		\end{cases}
		$$
		where $r=\frac{q}{q-1}$.
	\end{theorem}
	\begin{proof}
		As we know from \cite[Theorem 14.5]{DJT95} that for each $\epsilon>0$ there are $x_1\dots x_n\in X$ such that for all $z=(z_1,\dots,z_n)\in \mathbb{C}^n$ we have
		$$
		\frac{1}{1+\epsilon}\|z\|_\infty \leq \bigg\|\sum_{k=1}^{n}z_k x_k\bigg\|\leq \|z\|_q.
		$$
		In particular, taking $z=e_k$ we have
		$
		1 \leq (1+\epsilon) \|x_k\|.
		$
		Hence, for $\epsilon>0$ there exist $x_1\dots x_n\in X$ such that
		$$
		\frac{n}{1+\epsilon} \leq \sum_{k=1}^{n} \|x_k\|.
		$$
		Using the definition of $K_1(B_{{\ell}_p^n},X,1)$ we obtain
		$$
		\sum_{k=1}^{n} \|x_k\| \frac{1}{n^{1/p}}\leq \frac{1}{K_1(B_{{\ell}_p^n},X,1)}\sup_{z\in B_{{\ell}_p^n}}\bigg\|\sum_{k=1}^{n}z_k x_k\bigg\| \leq \frac{1}{K_1(B_{{\ell}_p^n},X,1)} \sup_{z\in B_{{\ell}_p^n}}\|z\|_{\mbox{Cot}(X)}'
		$$
		(Here we use the convention $n^\frac{1}{\infty}=1$). 
		
		Now if $p\leq \mbox{Cot}(X)$ then $\sup_{z\in B_{{\ell}_p^n}}\|z\|_{\mbox{Cot}(X)}\leq 1$. But if $ \mbox{Cot}(X)<p$ then we have
		$$
		\|z\|_{\mbox{Cot}(X)} \leq \|z\|_p n^{\frac{1}{\text{Cot}(X)}-\frac{1}{p}}.
		$$
		Thus we obtain 
		$$
		K_1(B_{{\ell}_p^n},X,1)\leq\begin{cases}\vspace{0.1cm}
			\cfrac{1}{n^{1-\frac{1}{p}}},
			& \text{if $p\leq\mbox{Cot}(X)$,}\\\vspace{0.2cm}
			\cfrac{1}{n^{1-\frac{1}{\text{Cot}(X)}}},
			& \text{if $\mbox{Cot}(X)<p$}.
		\end{cases}
		$$
		A simple observation provides
		$$
		K(B_{{\ell}_p^n},X,\bm{\lambda})\leq K_1(B_{{\ell}_p^n},X,\bm{\lambda})=\bm{\lambda} K_1(B_{{\ell}_p^n},X,1).
		$$
		This completes half of the proof.
		
		For the next half of the proof, we 
		consider $r>1$ such that $1/r+1/q=1$. Using H\"{o}lder's inequality, for all $z\in \mathbb{C}^n$ we get
		\begin{align*}
			\sum_{|\alpha|=m} \|c_\alpha z^\alpha \|&= \sum_{{\bf j}\in \mathcal{J}(m,n)} \mbox{card}[{\bf j}] \|A(e_{j_1},\cdots,e_{j_m})\||z_{j_1}\cdots z_{j_m}|\\
			&= \sum_{{\bf i}\in \mathcal{M}(m,n)} \|A(e_{i_1},\cdots,e_{i_M})\||z_{i_1}\cdots z_{i_m}|\\
			&\leq \bigg(\sum_{{\bf i}\in \mathcal{M}(m,n)} \|A(e_{i_1},\cdots,e_{i_M})\|^q \bigg)^{\frac{1}{q}}
			\bigg(\sum_{{\bf i}\in \mathcal{M}(m,n)} |z_{i_1}\cdots z_{i_m}|^r \bigg)^{\frac{1}{r}}.
		\end{align*}
		Using equation \eqref{KMeq3.3} we obtain
		\begin{align*}
			\sum_{|\alpha|=m} \|c_\alpha z^\alpha \|
			&\leq C_q(X)^m \|A\|_{B_{{\ell}_p^n}} \bigg(\sum_{{\bf i}\in \mathcal{M}(m,n)} |z_{i_1}\cdots z_{i_m}|^r \bigg)^{\frac{1}{r}}\\
			& =C_q(X)^m \|A\|_{B_{{\ell}_p^n}} (|z_1|^r+\cdots+|z_n|^r)^{\frac{m}{r}}.
		\end{align*}
		As we know that $\|z\|_r\leq\|z\|_p$ for $p\leq r$ then for all $z\in \frac{1}{C_q(X)}B_{{\ell}_p^n}$
		$$
		\sum_{|\alpha|=m} \|c_\alpha z^\alpha \| \leq \|A\|_{B_{{\ell}_p^n}} \leq \frac{m^m}{m!}\|P\|_{B_{{\ell}_p^n}} \leq e^m \|P\|_{B_{{\ell}_p^n}}.
		$$
		If $r<p$ then we have
		$$
		\|z\|_r \leq n^{\frac{1}{r}-\frac{1}{p}}\|z\|_p .
		$$
		Hence,
		$$
		\sum_{|\alpha|=m} \|c_\alpha z^\alpha \| \leq n^{\frac{m}{r}-\frac{m}{p}} C_q(X)^m \|A\|_{B_{{\ell}_p^n}}  \leq e^m n^{\frac{m}{r}-\frac{m}{p}} C_q(X)^m \|P\|_{B_{{\ell}_p^n}}.
		$$
		The above arguments show that, for every $m$, 
		$$
		K_m(B_{{\ell}_p^n},X,1)\geq \begin{cases}\vspace{0.1cm}
			\cfrac{1}{e},
			& \text{if $p\leq r$,}\\\vspace{0.2cm}
			\cfrac{1}{e n^{\frac{1}{r}-\frac{1}{p}} C_q(X)},
			& \text{if $r<p$}.
		\end{cases}
		$$
		After applying Lemma \ref{4Klemma1} (2) we finally obtain 
		$$
		K(B_{{\ell}_p^n},X,\bm{\lambda})\geq \begin{cases}\vspace{0.1cm}
			\cfrac{\bm{\lambda}-1}{\bm{\lambda}} \cdot \cfrac{1}{e},
			& \text{if $p\leq r$,}\\\vspace{0.2cm}
			\cfrac{\bm{\lambda}-1}{\bm{\lambda}} \cdot \cfrac{1}{e n^{\frac{1}{r}-\frac{1}{p}} C_q(X)},
			& \text{if $r<p$},
		\end{cases}
		$$
		for every $n>1$ and every $\bm{\lambda}>1$.
		This concludes the result.
	\end{proof}
	The following result gives the estimates of the arithmetic Bohr radius for infinite dimensional complex Banach space of cotype $q$, for $2\leq q\leq \infty$.
	\begin{theorem} \label{Aqcotype}
		Let $1\leq p \leq \infty$. Suppose $X$ is an infinite dimensional complex Banach space of cotype $q$ with $2\leq q\leq \infty$. Then we have
		$$
		A(B_{{\ell}_p^n},X,\bm{\lambda})\leq\begin{cases}\vspace{0.1cm}
			\cfrac{\bm{\lambda}}{n},
			& \text{if $p\leq\mbox{Cot}(X)$,}\\\vspace{0.2cm}
			\cfrac{\bm{\lambda}}{n^{1-\frac{1}{\text{Cot}(X)}+\frac{1}{p}}},
			& \text{if $\mbox{Cot}(X)<p$},
		\end{cases}
		$$
		and
		$$
		A(B_{{\ell}_p^n},X,\bm{\lambda})\geq \begin{cases}\vspace{0.1cm}
			\cfrac{\bm{\lambda}-1}{\bm{\lambda}} \cdot \cfrac{1}{en^{\frac{1}{p}}},
			& \text{if $p\leq r$,}\\\vspace{0.2cm}
			\cfrac{\bm{\lambda}-1}{\bm{\lambda}} \cdot \cfrac{1}{e n^{\frac{1}{r}} C_q(X)},
			& \text{if $r<p$},
		\end{cases}
		$$
		where $r=\frac{q}{q-1}$.
	\end{theorem}
	\begin{proof} The lower bound of $A(B_{{\ell}_p^n},X,\bm{\lambda})$ follows from Lemma \ref{KMthe1.4} and Theorem \ref{Bqcotype}.
		For upper bound we use the following fact from Theorem \ref{Bqcotype}: for $\epsilon>0$ there exist $x_1\dots x_n\in X$ such that 
		$$
		\frac{1}{(1+\epsilon)} \leq  \|x_k\|, \,\ 1\leq k\leq n.
		$$
		Let $s:=(s_1,\dots,s_n)\in\mathbb{R}_{\geq 0}^n$ be such that 
		$$
		\sum_{k=1}^n \|c_k(P)\| s_k \leq \sup_{z\in B_{\ell_p^n}} \big\|P(z)\big\|, \ \mbox{ for every } P\in \mathcal{P}(\prescript{1}{}{\ell_p^n}, X)
		$$
		where $X$ is an infinite dimensional complex Banach space of cotype $q$ with $2\leq q\leq \infty$. Then we have
		$$
		\frac{n}{(1+\epsilon)}\frac{\sum_{k=1}^{n}s_k}{n} \leq\sum_{k=1}^{n} \|x_k\| s_k\leq \sup_{z\in B_{{\ell}_p^n}}\bigg\|\sum_{k=1}^{n}z_k x_k\bigg\| \leq \sup_{z\in B_{{\ell}_p^n}}\|z\|_{\mbox{Cot}(X)}.
		$$
		Thus we obtain 
		$$
		A(\mathcal{P}(\prescript{1}{}{\ell_p^n},X),X,1)\leq\begin{cases}\vspace{0.1cm}
			\cfrac{1}{n},
			& \text{if $p\leq\mbox{Cot}(X)$,}\\\vspace{0.2cm}
			\cfrac{1}{n^{1-\frac{1}{\text{Cot}(X)}+\frac{1}{p}}},
			& \text{if $\mbox{Cot}(X)<p$}.
		\end{cases}
		$$
		As we know from Proposition \ref{Prop2.5}
		$$
		A(B_{{\ell}_p^n},X,\bm{\lambda})\leq A(\mathcal{P}(\prescript{1}{}{\ell_p^n},X),X,\bm{\lambda})=\bm{\lambda} A(\mathcal{P}(\prescript{1}{}{\ell_p^n},X),X,1).
		$$
		This concludes the proof.
	\end{proof}
	\begin{remark}
		Alternatively, Theorem \ref{Aqcotype} and Lemma \ref{KMthe1.4} together provide the upper bound of $K(B_{{\ell}_p^n},X,\bm{\lambda})$ in Theorem \ref{Bqcotype}. 
	\end{remark}
	Before proceeding further we required the following result which is an immediate consequence of \cite[Theorem 5.3]{DMS12}: Given $ P\in \mathcal{P}(\prescript{m}{}{\ell_p^n}, X)$ and $z\in B_{{\ell}_p^n}$, use the result \cite[Theorem 5.3]{DMS12} to the polynomial $Q:=P(D_z)$, where $D_z: \ell_\infty^n \to \ell_p^n$ is a diagonal operator. 
	Thus we obtain the following result.  
	\begin{theorem} \label{Oqconcave}
		Let $2\leq q<\infty$. Suppose $Y$ is a q-concave Banach lattice, and the operator $\mathcal{V}: X\to Y$ is an $(r, 1)$-summing with $r \in [1,q]$. Let
		$$
		\rho:=\frac{qrm}{q+(m-1)r}.
		$$
		Then there is a constant $C>0$ such that the following holds
		$$
		\bigg(\sum_{|\alpha|=m}\|\mathcal{V}(c_\alpha) z^\alpha\|_Y^\rho\bigg)^{\frac{1}{\rho}} \leq C^m \|P\|_{\mathcal{P}(\prescript{m}{}{\ell_p^n})},
		$$
		for every polynomial $ P\in \mathcal{P}(\prescript{m}{}{\ell_p^n}, X)$,  given by $P(z)=\sum_{|\alpha|=m}c_\alpha z^\alpha$.
	\end{theorem}
	It is worth mentioning here that \cite[Theorem 5.3]{DMS12} even holds for Banach spaces of cotype $q$, (see, \cite{CMS20}). Therefore, Theorem \ref{Oqconcave} is also holds for  Banach spaces of cotype $q$. 
	Now we shall give proof of the following explicit version of Theorem \ref{KMthe3.31}, which is already discussed in \cite[Theorem 5.4]{DMS12} for the case $p=\infty$.
	\begin{theorem}\label{KMthe3.3}
		Suppose the operator $\mathcal{V}:X \to Y$ is a bounded operator between Banach spaces $X$ and $Y$. Then the following holds.
		\begin{enumerate}
			\item Suppose that $X$ or $Y$ is of cotype $q$, where $2\leq q \leq \infty$. Then there is a constant $C>0$ such that the following inequality holds
			$$
			K(B_{{\ell}_p^n},\mathcal{V},\bm{\lambda})\geq \begin{cases}\vspace{0.1cm}
				C\cfrac{\bm{\lambda}-\|\mathcal{V}\|}{\bm{\lambda}},
				& \text{if $p\leq r$,}\\\vspace{0.2cm}
				C\cfrac{\bm{\lambda}-\|\mathcal{V}\|}{\bm{\lambda}} \cdot \cfrac{1}{ n^{\frac{1}{r}-\frac{1}{p}}},
				& \text{if $r<p$},
			\end{cases}
			$$
			for every $\|\mathcal{V}\|<\bm{\lambda}$ and for every $n$, where  $r=q/(q-1)$.
			\item Suppose that $Y$ is a $q$-concave Banach lattice, where $2\leq q \leq \infty$ and there is a $r \in [1,q)$ such that the operator $\mathcal{V}$ is $(r,1)$-summing. Then there is a constant $C>0$ such that  the following holds
			$$
			K(B_{{\ell}_p^n},\mathcal{V},\bm{\lambda})\geq C \frac{\bm{\lambda}-\|\mathcal{V}\|}{2\bm{\lambda}-\|\mathcal{V}\|} \bigg( \frac{\log n}{n} \bigg)^{1-\frac{1}{q}},
			$$
			for every $\|\mathcal{V}\|<\bm{\lambda}$ and for every $n$.
		\end{enumerate}
	\end{theorem}
	\begin{proof}
		We can easily obtain (1) with the help of the following observation 
		\begin{align*}
			K(B_{{\ell}_p^n},\mathcal{V},\bm{\lambda}) \geq \max \{K(B_{{\ell}_p^n},X,\bm{\lambda}/\|\mathcal{V}\|),K(B_{{\ell}_p^n},Y,\bm{\lambda}/\|\mathcal{V}\|)\}  
		\end{align*}
		and Theorem \ref{Bqcotype} gives the desired result.
		
		\medskip
		For the proof of (2) we use Theorem \ref{Oqconcave} and H\"{o}lder's inequality, and obtain that, for every polynomial $P(z)=\sum_{|\alpha|=m} c_\alpha z^\alpha\in \mathcal{P}(\prescript{m}{}{\ell_p^n}, X)$
		\begin{align*}
			\sum_{|\alpha|=m}\|\mathcal{V}(c_\alpha) z^\alpha\|_Y & \leq \bigg(\sum_{|\alpha|=m} 1 \bigg)^{\frac{(q-1)mr-q+r}{qmr}} \times \bigg(\sum_{|\alpha|=m}\|\mathcal{V}(c_\alpha) z^\alpha\|_Y^{\frac{qmr}{q+(m-1)r}} \bigg)^{\frac{q+(m-1)r}{qmr}}\\
			&\leq \bigg(\sum_{|\alpha|=m} 1 \bigg)^{\frac{(q-1)mr-q+r}{qmr}} C^m \|P\|_{\mathcal{P}(\prescript{m}{}{\ell_p^n})}.
		\end{align*}
		So by the following fact
		$$
		K_m(B_{\ell_p^n},\mathcal{V},\bm{\lambda})=\bm{\lambda}^{1/m}K_m(B_{\ell_p^n},\mathcal{V},1),
		$$
		we obtain that 
		$$
		K_m(B_{\ell_p^n},\mathcal{V},\bm{\lambda})\geq E \bm{\lambda}^{1/m} \bigg(1+\frac{n}{m}\bigg)^{-\frac{(q-1)mr-q+r}{qmr}}.
		$$
		Hence, finally, we have
		$$
		\inf_m\Big\{K_m(B_{\ell_p^n},\mathcal{V},\bm{\lambda})\Big\}\geq C \bigg( \frac{\log n}{n} \bigg)^{1-\frac{1}{q}}.
		$$
		Finally, Lemma \ref{4Klemma1} (1) gives the conclusion.
	\end{proof}
	\begin{remark}
		The lower bound of the arithmetic Bohr radius $A(B_{{\ell}_p^n},\mathcal{V},\bm{\lambda})$ with the given conditions in Theorem \ref{KMthe3.3} can simply obtain using Lemma \ref{KMthe1.4}.
	\end{remark}
	The following result is talk about the Bohr radii of embeddings. We remark that the special choice $p=\infty$ takes this result to \cite[Theorem 1.4]{DMS12}.
	\begin{theorem} \label{Bimbe}
		Let $1\leq r <q<\infty$. Then we have 
		$$
		K(B_{{\ell}_p^n},\ell_r \hookrightarrow \ell_q,\bm{\lambda})\lesssim \begin{cases}\vspace{0.1cm}
			\Bigg(\cfrac{\log n}{n}\Bigg)^{1-\frac{1}{\min\{p,2\}}},
			& \text{if $1\leq r \leq 2$}\\\vspace{0.2cm}
			\cfrac{1}{n^{1-\frac{1}{p}}},
			& \text{if $p\leq r$ and $2\leq r<\infty$}\\\vspace{0.2cm}
			\cfrac{1}{n^{1-\frac{1}{r}}},
			& \text{if $r<p$ and $2\leq r<\infty$}.
		\end{cases}
		$$
		and
		$$
		K(B_{{\ell}_p^n},\ell_r \hookrightarrow \ell_q,\bm{\lambda})\gtrsim \begin{cases}\vspace{0.1cm}
			\sqrt{\cfrac{\log n}{n}},
			& \text{if $1\leq r \leq 2$}\\\vspace{0.2cm}
			\cfrac{1}{e},
			& \text{if $p\leq r$ and $2\leq r<\infty$}\\\vspace{0.2cm}
			\cfrac{1}{n^{1-\frac{1}{r}-\frac{1}{p}}},
			& \text{if $r<p$ and $2\leq r<\infty$},
		\end{cases}
		$$
		with constants depending only on $\bm{\lambda}$ and $r,q$.
	\end{theorem}
	\begin{proof}
		The case $1\leq r \leq 2$ can be handle using following observation and Theorem \ref{maintheorem}
		$$
		K(B_{{\ell}_p^n},\mathcal{V},\bm{\lambda})\leq K(B_{{\ell}_p^n},\mathbb{C},\bm{\lambda})\lesssim \Bigg(\frac{\log n}{n}\Bigg)^{1-\frac{1}{\min\{p,2\}}}.
		$$
		Now, for the case $2\leq r<\infty$, we establish the upper bound of $K(B_{{\ell}_p^n},\ell_r \hookrightarrow \ell_q,\bm{\lambda})$. The definition of $K_1(B_{{\ell}_p^n},\ell_r \hookrightarrow \ell_q,1)$ gives that, for each $n$,
		
		\begin{align*}
			n^{1-\frac{1}{p}}&=\sum_{k=1}^{n} \|e_k\|_q~\frac{1}{n^{1/p}} \leq \frac{1}{K_1(B_{{\ell}_p^n},\ell_r \hookrightarrow \ell_q,1)}\sup_{z\in B_{{\ell}_p^n}}\bigg\|\sum_{k=1}^{n}z_k e_k\bigg\|_r\\
			&\leq \frac{1}{K_1(B_{{\ell}_p^n},\ell_r \hookrightarrow \ell_q,1)}\sup_{z\in B_{{\ell}_p^n}}\|z\|_r.
		\end{align*}
		Now, we obtain
		$$
		K_1(B_{{\ell}_p^n},\ell_r \hookrightarrow \ell_q,\bm{\lambda})\leq \begin{cases}\vspace{0.1cm}
			\cfrac{\bm{\lambda}}{n^{1-\frac{1}{p}}},
			& \text{if $p\leq r$,}\\\vspace{0.2cm}
			\cfrac{\bm{\lambda}}{n^{1-\frac{1}{r}}},
			& \text{if $r<p$}.
		\end{cases}
		$$
		The fact
		$$
		K(B_{{\ell}_p^n},\ell_r \hookrightarrow \ell_q,\bm{\lambda})\leq K_1(B_{{\ell}_p^n},\ell_r \hookrightarrow \ell_q,\bm{\lambda})
		$$
		concludes the lower bounds.
		
		We use the following three different cases to determine lower bounds:
		
		\medskip
		\noindent
		The case $1\leq r<q\leq 2$. It is well known from the Bennett-Carl theorem \cite{B73,C74} that the inclusion $\ell_r \hookrightarrow \ell_q$ is $(s,1)$-summing where $$
		\frac{1}{s}=\frac{1}{2}+\frac{1}{r}-\max \bigg\{\frac{1}{q},\frac{1}{2}\bigg\}.
		$$
		Since $\ell_q$ is $2$-concave, the lower estimate follows from Theorem \ref{KMthe3.3} and hence we obtain
		$$
		K(B_{{\ell}_p^n},\ell_r \hookrightarrow \ell_q,\bm{\lambda})\gtrsim \sqrt{\frac{\log n}{n}}
		$$
		The above relation also holds for $1\leq r <2 \leq q$. Because $K(B_{{\ell}_p^n},\ell_r \hookrightarrow \ell_q,\bm{\lambda})\geq K(B_{{\ell}_p^n},\ell_r \hookrightarrow \ell_2,\bm{\lambda})$. 
		
		\medskip
		\noindent
		The case $2\leq r$. Note that $K(B_{{\ell}_p^n},\ell_r \hookrightarrow \ell_q,\bm{\lambda})\geq K(B_{{\ell}_p^n},\ell_r \hookrightarrow \ell_r,\bm{\lambda})$. Then Theorem \ref{Bqcotype} concludes the result.
	\end{proof}
	To provide the proof of Theorem \ref{th1.2}, we require the result of Kwapie\'{n} \cite{K68}: it states that every operator $\mathcal{V}: \ell_1 \to \ell_q$ is $(r,1)$-summing, where
	$$
	\frac{1}{r}=1-\bigg|\frac{1}{q}-\frac{1}{2}\bigg|.
	$$
	\begin{proof}[Proof of Theorem \ref{th1.2}]
		The upper bound we can estimate from the proof of Theorem \ref{Bimbe}. For the lower bound, we use Theorem \ref{KMthe3.3}, Kwapie\'{n} theorem \cite{K68}, and the well-known fact that $\ell_q$ is $\max\{2,q\}$-concave.
	\end{proof}
	Now we estimate the upper and lower bounds of the arithmetic Bohr radii of embeddings.
	\begin{theorem}\label{Aembed}
		Let $1\leq r <q<\infty$. Then we have 
		$$
		A(B_{{\ell}_p^n},\ell_r \hookrightarrow \ell_q,\bm{\lambda})\lesssim \begin{cases}\vspace{0.1cm}\cfrac{\big(\log n\big)^{ 1- (1/\min\{p,2\})}}{\displaystyle n^{ \frac{1}{2}+(1/\max\{p,2\})}}
			,
			& \text{if $1\leq r \leq 2$}\\\vspace{0.2cm}
			\cfrac{1}{n},
			& \text{if $p\leq r$ and $2\leq r<\infty$}\\\vspace{0.2cm}
			\cfrac{1}{n^{1-\frac{1}{r}+\frac{1}{p}}},
			& \text{if $r<p$ and $2\leq r<\infty$}.
		\end{cases}
		$$
		and
		$$
		A(B_{{\ell}_p^n},\ell_r \hookrightarrow \ell_q,\bm{\lambda})\gtrsim \begin{cases}\vspace{0.1cm}
			\cfrac{\sqrt{\log n}}{n^{1+\frac{1}{p}}},
			& \text{if $1\leq r \leq 2$}\\\vspace{0.2cm}
			\cfrac{1}{en^{\frac{1}{p}}},
			& \text{if $p\leq r$ and $2\leq r<\infty$}\\\vspace{0.2cm}
			\cfrac{1}{n^{1-\frac{1}{r}}},
			& \text{if $r<p$ and $2\leq r<\infty$},
		\end{cases}
		$$
		with constants depending only on $\bm{\lambda}$ and $r,q$.
	\end{theorem}
	\begin{proof} The lower bound of $A(B_{{\ell}_p^n},\ell_r \hookrightarrow \ell_q,\bm{\lambda})$ is consequence of Lemma \ref{KMthe1.4} and Theorem \ref{Bimbe}.
		For the upper bound the case $1\leq r \leq 2$ can be handled using the following observation and Theorem \ref{maintheorem}
		$$
		A(B_{{\ell}_p^n},\mathcal{V},\bm{\lambda})\leq A(B_{{\ell}_p^n},\mathbb{C},\bm{\lambda})\lesssim \cfrac{\big(\log n\big)^{ 1- (1/\min\{p,2\})}}{\displaystyle n^{ \frac{1}{2}+(1/\max\{p,2\})}}.
		$$
		Now, for the case $2\leq r<\infty$, we establish the upper bound of $A(B_{{\ell}_p^n},\ell_r \hookrightarrow \ell_q,\bm{\lambda})$. Let $s:=(s_1,\dots,s_n)\in\mathbb{R}_{\geq 0}^n$ such that 
		$$
		\sum_{k=1}^n \|\mathcal{V}(c_k(P))\|_q ~ s_k \leq \sup_{z\in B_{\ell_p^n}} \big\|P(z)\big\|_r, \ \mbox{ for every } P\in \mathcal{P}(\prescript{1}{}{\ell_p^n}, \ell_r^n)
		$$
		where $\mathcal{V}: \ell_r \hookrightarrow \ell_q$ with $2\leq r\leq \infty$. Then we have
		
		$$
		\sum_{k=1}^{n} s_k=\sum_{k=1}^{n} s_k \|e_k\|_q \leq \sup_{z\in B_{{\ell}_p^n}}\bigg\|\sum_{k=1}^{n}z_k e_k\bigg\|_r
		\leq \sup_{z\in B_{{\ell}_p^n}}\|z\|_r.
		$$
		Now, we obtain
		$$
		A(\mathcal{P}(\prescript{1}{}{\ell_p^n}, \ell_r^n),\ell_r \hookrightarrow \ell_q,\bm{\lambda})\leq \begin{cases}\vspace{0.1cm}
			\cfrac{\bm{\lambda}}{n^{1-\frac{1}{p}}},
			& \text{if $p\leq r$,}\\\vspace{0.2cm}
			\cfrac{\bm{\lambda}}{n^{1-\frac{1}{r}}},
			& \text{if $r<p$}.
		\end{cases}
		$$
		Since
		$$
		A(B_{{\ell}_p^n},\ell_r \hookrightarrow \ell_q,\bm{\lambda})\leq A(\mathcal{P}(\prescript{1}{}{\ell_p^n}, \ell_r^n),\ell_r \hookrightarrow \ell_q,\bm{\lambda}).
		$$
		Hence we have the lower bound.
	\end{proof}
	\begin{remark}
		We can use Theorem \ref{Aembed} and Lemma \ref{KMthe1.4} together as a tool to find the upper estimate of $K(B_{{\ell}_p^n},\ell_r \hookrightarrow \ell_q,\bm{\lambda})$ in Theorem \ref{Bimbe}. 
	\end{remark}
	\begin{proof}[Proof of Theorem \ref{Aoperator}] The lower bound obtain from Theorem \ref{th1.2} and Lemma \ref{KMthe1.4}, and the upper bound is from Theorem \ref{Aembed}.
	\end{proof}

	\bigskip
	\noindent
	{\bf Acknowledgement.} 
	The work of the first author is supported by the institute post doctoral fellowship of NISER Bhubaneswar, India. He thanks the National Institute of Science Education and Research Bhubaneswar, for providing an excellent research facility. The second author is thankful for the research grants (DST/INSPIRE/04/2019/001914).

\end{document}